\documentclass[journal,twoside,web]{ieeecolor}  % Comment this line out if you need a4paper
\usepackage{generic}
\usepackage{textcomp}
\usepackage{hyperref}       % hyperlinks
\usepackage{url}            % simple URL typesetting
\usepackage{amsfonts}       % blackboard math symbols
\usepackage{xcolor}

\usepackage{amsmath, amssymb, mathtools, amsthm}
\usepackage{graphicx}
\usepackage{algorithm, algorithmic}
\usepackage{bm}
\usepackage{float}
\usepackage{color}
\usepackage{epstopdf}
\usepackage{hyperref}
\newtheorem{Theorem}{Theorem}
\newtheorem{Lemma}{Lemma}

\newtheorem{Remark}{Remark}
\newtheorem{Cor}{Corollary}
\newtheorem{Assumption}{Assumption}

\newcommand{\sgn}{\textup{sgn}}
\newcommand{\sign}{\textup{sign}}
\makeatletter
\newcommand\footnoteref[1]{\protected@xdef\@thefnmark{\ref{#1}}\@footnotemark}
\makeatother

\DeclareRobustCommand{\rchi}{{\mathpalette\irchi\relax}}
\newcommand{\irchi}[2]{\raisebox{\depth}{$#1\chi$}}
\def\BibTeX{{\rm B\kern-.05em{\sc i\kern-.025em b}\kern-.08em
		T\kern-.1667em\lower.7ex\hbox{E}\kern-.125emX}}
\begin{document}
	
	\title{Robust Distributed Fixed-Time Economic Dispatch under Time-Varying Topology}
	
	\author{Mayank~Baranwal$^a$, Kunal~Garg$^b$, \IEEEmembership{Student Member, IEEE}, Dimitra~Panagou$^{b}$, \IEEEmembership{Senior Member, IEEE}, and~Alfred~O.~Hero$^{c}$, \IEEEmembership{Fellow, IEEE}
		\thanks{The authors are with the $^a$Division of Data \& Decision Sciences, Tata Consultancy Services, Mumbai, India; Departments of $^b$Aerospace Engineering, $^c$Electrical Engineering and Computer Science, University of Michigan, Ann Arbor, MI 48109, USA. e-mail(s): \texttt{ baranwal.mayank@tcs.com, kgarg@umich.edu, dpanagou@umich.edu, hero@umich.edu}.}}
	
	\maketitle
	\thispagestyle{empty}
	\begin{abstract}
		The centralized power generation infrastructure that defines the North American electric grid is slowly moving to the distributed architecture due to the explosion in use of renewable generation and distributed energy resources (DERs), such as residential solar, wind turbines and battery storage. Furthermore, variable pricing policies and profusion of flexible loads entail frequent and severe changes in power outputs required from the individual generation units, requiring fast availability of power allocation. To this end, a \emph{fixed-time} convergent, fully distributed  economic dispatch algorithm for scheduling optimal power generation among a set of DERs is proposed. The proposed algorithm incorporates both load balance and generation capacity constraints.
	\end{abstract}
	
	
	\vspace{-1em}\section{Introduction}\label{sec:Intro}
	\IEEEPARstart{E}{conomic} dispatch (ED) is one of the key optimization issues in power systems, and concerns with optimal allocation of power output from a number of generators in order to meet the system load requirements at the lowest possible cost, subject to operation constraints on generators \cite{wood2013power}. Various analytical and heuristic techniques have been proposed to address the ED problem including, but not limited to, Newton-Raphson gradient descent~\cite{lin1992direct}, using optimized transition matrix~\cite{yan2019consensus}, and estimating power mismatch~\cite{pourbabak2017novel}. However, these methods address the ED problem in a \emph{centralized} manner, where a global control center processes the information and implements the centralized dispatch algorithms, requiring access to global quantities, such as load and generator output values of each node in the network. While centralized architectures offer easy implementation of ED algorithms, they are vulnerable to single-point failures. In addition, centralized ED algorithms do not scale well with the number of generators and need restructuring as the power system evolves with time \cite{chen2016distributed}. To increase robustness, scalability and efficiency, the centralized power generation infrastructure is slowly moving towards a distributed implementation. As a consequence, several distributed dispatch algorithms have been proposed in the recent years \cite{chen2016distributed, feng2017finite}. 
	Distributed architectures avoid single-point failures and offer plug-and-play capabilities, where a DER can be added or removed from an existing power system in a communication agnostic fashion.
	
	An ED problem with load balance constraint can be formulated as a constrained optimization problem, characterized by a Lagrange multiplier corresponding to the constraint. This Lagrange multiplier is often referred as the incremental cost, and must be equal for all generators at optimality \cite{zhang2012convergence}. Thus a centralized ED problem can be addressed in a distributed manner by reaching consensus on incremental cost associated with each generation unit. Several consensus based approaches, namely the incremental cost consensus (ICC) \cite{zhang2012convergence}, ratio consensus method \cite{binetti2014distributed}, and distributed gradient method \cite{zhang2014online} have been proposed as viable alternatives of the centralized ED method. While these methods do alleviate some of the major issues associated with centralized ED algorithms, they have their own disadvantages. For instance, the algorithm in \cite{yang2013consensus} requires global information about power output from each generator, as well as the total load demand. 
	% Thus, the method is not fully distributed. Similarly, the value of step-size, particularly near the constraint bounds, plays a significant role with regards to convergence behavior of distributed gradient method. Moreover, each of the aforementioned distributed ED algorithms only guarantee asymptotic convergence.
	
	As the power systems become more complex due to increased penetration of DERs, flexible loads and dynamic pricing, power outputs required from every generation units undergo frequent and severe changes, and thus it is crucial to investigate distributed ED algorithms with fast convergence characteristics. Recent works, such as \cite{chen2016distributed, feng2017finite}, investigate distributed nonlinear protocols that guarantee consensus on incremental costs associated with each generation unit, as well as its convergence to optimal solution in a finite time. The notion of finite-time stability of dynamical systems was introduced by the authors in \cite{bhat2000finite}, guaranteeing convergence to the equilibrium point within a finite time. The algorithms proposed in \cite{chen2016distributed, feng2017finite} are based either on the finite-time consensus protocol \cite{wang2010finite} or the finite-time average consensus algorithm (FACA) \cite{kibangou2012graph}. However, convergence time, even though finite, depends on the initial values of the individual incremental costs, and increase as in the initial conditions go farther away from the equilibrium point. Fixed-time convergence \cite{polyakov2012nonlinear} is a stronger notion of convergence, where convergence-time does not depend upon the initial values of the incremental costs. In this paper, a distributed fixed-time algorithm is proposed or time-varying communication graphs and additive uncertainties.
	% In contrast to existing work on finite-time ED algorithm, t
	The key contributions of this paper are:\\
	{\bf Fixed-time convergence}: A novel fixed-time consensus algorithm is proposed, and employed to solve large-scale distributed ED problem, within a user-specified fixed-time.\\
	{\bf Time-varying communication topology}: Different from the physical architecture of the power system and in contrast to most of the aforementioned work on finite-time approaches, the proposed framework allows for a separate communication network, the topology of which is allowed to vary with time.\\
	{\textbf{Robustness to additive disturbances}:
		The fixed-time consensus algorithm developed in this paper is designed to be robust with respect to a class of additive disturbances.}
	\\
	{\bf Consistent discretization}: 
	% While optimization methods in continuous-time have major theoretical relevance, sampling constraints may preclude continuous-time acquisition and updating  \cite{chen2016distributed, feng2017finite}. 
	A rate-matching discretization scheme is discussed, which allows the mentioned convergence properties to be preserved for discretized implementation.

	\vspace{-.5em}\section{Preliminaries}\label{sec:prelim}
	
	\vspace{-.5em}\subsection{Notation}
	We use $\mathbb{R}$, $\mathbb R_+$ to denote the set of reals and non-negative reals, respectively. $\mathcal{G}=(A,\mathcal{V})$ represents an undirected graph with adjacency matrix $A = [a_{ij}]\in \mathbb R^{N\times N}$, $a_{ij}\in\{0,1\}$ and set of nodes $\mathcal{V} = \{1, 2, \cdots, N\}$. $\mathcal{N}_i$ represents the set of 1-hop neighbors of node $i$. In particular, it follows that $\sum_{j=1}^Na_{ij}f_j(\cdot)=\sum_{j\in\mathcal{N}_i}f_j(\cdot)$ for any function $f_j(\cdot)$. $\Lambda_2(\cdot)$ represents the second smallest eigenvalue of a matrix. Finally, for any $x\in\mathbb{R}$, we define function $\sgn^\mu: \mathbb R\rightarrow\mathbb R$ as: $\sgn^\mu(x) = |x|^{\mu}\sign(x)$, $\mu>0$, with $\sgn(x)\triangleq\sign(x)$.
	
	\vspace{-1em}\subsection{Fixed-time stability (FxTS)}\label{subsec:FxTS}
	Consider the following dynamical system: 
	{\small\begin{align}\label{eq:sys}
			\dot x(t) = f(x(t)),
	\end{align}}\normalsize
	where $x\in \mathbb R^d$, $f: \mathbb R^d \rightarrow \mathbb R^d$ is a continuous function with $f(0)=0$. As defined in \cite{polyakov2012nonlinear}, the origin is said to be a fixed-time stable (FxTS) equilibrium of \eqref{eq:sys} if it is Lyapunov stable and $\lim_{t\to T} x(t)=0$, where $T<\infty$ is independent of $x(0)$.
	\begin{Lemma}[\cite{polyakov2012nonlinear}]\label{lemma:FxTS}
		Suppose there exists a positive definite function $V$ for system \eqref{eq:sys} such that $\dot V(x(t)) \leq -a(V(x(t)))^p-b(V(x(t)))^q$ with $a,b>0$, $0<p<1$ and $q>1$. Then, the origin of \eqref{eq:sys} is FxTS with settling time function satisfying $T \leq \frac{1}{a(1-p)} + \frac{1}{b(q-1)}$.
	\end{Lemma}
	
	\vspace{-1em} \subsection{Overview of graph theory}\label{subsec:graph}
	This subsection presents some Lemmas from graph theory and other inequalities that will be useful later. 
	\begin{Lemma}[\cite{hardy1988inequalities}]\label{lemma:ineq}
		Let $t_i\geq 0$ for $i\in \{1,2,\cdots, N\}$, then:
		\begin{subequations}\label{eq:ineq}
			{\small\begin{align}
					\sum\nolimits_{i = 1}^N t_i^p& \geq \Big(\sum\nolimits_{i = 1}^Nt_i\Big)^p, \; 0<p\leq 1,\\
					\sum\nolimits_{i = 1}^N t_i^p& \geq N^{1-p}\Big(\sum\nolimits_{i = 1}^Nt_i\Big)^p, \; p>1.
			\end{align}}\normalsize
		\end{subequations}
	\end{Lemma}
	
	\begin{Lemma}\label{lemma:sign}
		\footnote{\label{note1} The proof is a simple consequence of the fact that the function $\sign(x)$ is odd symmetric about zero.}Let $\mathcal G$ be a graph consisting of $N$ nodes and $x_i\in \mathbb R^d$ for $i\in \{1,2,\cdots, N\}$ and $\mathcal N_i$ denotes the in-neighbors of node $i$. Then, $\sum_{i = 1}^N\sum_{j \in \mathcal N_i}\sign(x_i-x_j) = 0.$
	\end{Lemma}
	
	\begin{Lemma}\label{lemma:eij_half}
		\footnoteref{note1}Let $w:\mathbb R^d\rightarrow \mathbb R^d$ be an odd function, i.e., $w(x) = -w(-x)$ for all $x\in \mathbb R^d$ and let the graph $\mathcal G = (A,\mathcal V)$ be undirected. Let $\{x_i\}$ and $\{e_i\}$ be the sets of vectors with $i\in \mathcal V$ and $x_{ij} \triangleq x_i-x_j$ and $e_{ij} \triangleq e_i-e_j$. Then, the following holds\vspace{-.5em}
		{\small\begin{align}\label{eq:f_ij_eij}
				\sum\nolimits_{i,j= 1}^Na_{ij}e_i^Tw(x_{ij}) = \frac{1}{2}\!\sum\nolimits_{i,j= 1}^Na_{ij}e_{ij}^Tw(x_{ij}).
		\end{align}}\normalsize
	\end{Lemma}
	
	\begin{Lemma}[\cite{mesbahi2010graph}]\label{lemma:Laplacian}
		Let $\mathcal G = (A,\mathcal V)$ be an undirected, connected graph. Let $L_A = [l_{ij}]\in\mathbb{R}^{N\times N}$ be its Laplacian matrix defined as $
		l_{ij} = \left\{
		\begin{array}{cc}
		\sum\limits_{k=1,k\neq i}^{N}a_{ik}, & i=j\\
		-a_{ij}, & i\neq j
		\end{array}.
		\right.    
		$
		Then, Laplacian $L_A$ has following properties:\\
		1) $L_A 1_N = 0_N$ and $\Lambda_2(L_A) > 0$.\\
		2) $x^TL_Ax=0.5\sum_{i,j=1}^Na_{ij}(x_j-x_i)^2$.
	\end{Lemma}
	
	\vspace{-1.5em}\subsection{Problem formulation}\label{subsec:prob_form}
	This work concerns with finding optimal power dispatch from a network of $N$ generators in a smart grid, under load balance (equality) and generation (inequality) constraints. Let $C_i(\cdot)$ be the cost associated with power generation for the $i^\text{th}$ generator. The traditional ED problem is described as~\cite{chen2016distributed}:\vspace{-.5em}
	{\small\begin{align}\label{eq:EDP}
			\min\limits_{\{P_i\}} \ \ \ \ &\sum\limits_{i=1}^{N}\underbrace{\alpha_iP_i^2 + \beta_iP_i + \gamma_i}_{C_i(P_i)}, \nonumber \\
			\text{subject to} \ \  &\sum\limits_{i=1}^{N}P_i = \sum\limits_{j=1}^{m}P_{\text{L}_j} = P_\text{tot},  \\
			&P_i^\text{min} \leq P_i \leq P_i^\text{max}, \nonumber
	\end{align}}\normalsize
	where $P_i, P_i^\text{min}$ and $P_i^\text{max}$ denote the power dispatched, minimum generation capability and maximum generation capability of the $i^\text{th}$ generator, respectively. $P_\text{tot}$ is the total power demanded by a network of $m$ loads, whereas $P_{\text{L}_j}$ indicates power requirement from the $j^\text{th}$ load. Here, $\alpha_i, \beta_i, \gamma_i >0$ are the cost coefficients associated with the $i^\text{th}$ generator. The equality constraint ensures that the total power generated by all the sources meets the total load power requirement.
	
	\vspace{-1em}\subsection{Optimal solution without generation constraints}\label{subsec:without}
	To gain relevant insights into the role of incremental cost, we first consider the problem of ED without generation constraints. The Lagrangian $L(\cdot)$ associated with ED problem in this case can be formulated as:
	{\small\begin{align}\label{eq:Lagrangian_without}
			L(\{P_i\},\lambda) = \sum\nolimits_{i=1}^{N}C_i(P_i) + \lambda\left(P_\text{tot}-\sum\nolimits_{i=1}^{N}P_i\right),
	\end{align}}\normalsize
	where $\lambda$ is the incremental cost or the Lagrange multiplier associated with the equality constraint. Let $\{P_i^{*}\}$ and $\lambda^*$ denote the optimal dispatch and incremental cost, respectively. Then the first-order condition of optimality yields:
	{\small\begin{align}\label{eq:P_star_without}
			P_i^* &= \frac{\lambda^*-\beta_i}{2\alpha_i}\geq 0, \ \ i=1,2,\dots,N.
	\end{align}}\normalsize
	The optimal incremental cost $\lambda^*$ can be obtained from the equality constraint as:
	{\small\begin{align}\label{eq:lambda_without}
			\lambda^* = \left(P_\text{tot}+\sum\nolimits_{i=1}^N\frac{\beta_i}{2\alpha_i}\right)\Big/\left(\sum\nolimits_{i=1}^N\frac{1}{2\alpha_i}\right).
	\end{align}}\normalsize
	In a distributed setting, different sources seek to estimate optimal incremental cost $\lambda^*$. Once the optimal incremental cost is known, optimal power dispatch for each generator can be obtained using \eqref{eq:P_star_without}.
	
	\vspace{-1em}\subsection{Optimal solution with generation constraints}\label{subsec:with}
	When generation limits are considered, then the optimal dispatch $\{P_i^*\}$ and incremental cost $\lambda^*$ satisfy the following relationship :\vspace{-.5em}
	{\small\begin{align}\label{eq:P_star_with}
			\left\{
			\begin{array}{cl}
				2\alpha_iP_i^* + \beta_i = \lambda^*, & \text{for} \ P_i^\text{min}<P_i^*<P_i^\text{max}, \\
				2\alpha_iP_i^* + \beta_i < \lambda^*, & \text{for} \ P_i^*=P_i^\text{max}, \\
				2\alpha_iP_i^* + \beta_i > \lambda^*, & \text{for} \ P_i^*=P_i^\text{min}.
			\end{array}
			\right.
	\end{align}}\normalsize
	If $P_i^\text{max}\to\infty$ and $P_i^\text{min}\to0$ for all $i$, i.e., if there are no generation constraints, then the optimal incremental cost satisfies the equality constraint in \eqref{eq:P_star_with}, which is identical to \eqref{eq:P_star_without} for the uncapacitated case. However, in the presence of generation constraints, \eqref{eq:P_star_without} does not provide the correct optimal solution for \eqref{eq:EDP}. 
	However, optimal incremental costs for the uncapacitated ED problem and \eqref{eq:EDP} are related as follows. Let $\Theta$ be the set of generators for which saturated optimal dispatch values, i.e., $P_i^*=P_i^\text{min}$ or $P_i^*=P_i^\text{max}$ for all $i\in\Theta$. Then, from \eqref{eq:P_star_with}, it follows that:\vspace{-.5em}
	{\small\begin{align}\label{eq:lambda_star_with}
			\lambda^* &= 2\alpha_iP_i^* + \beta_i, \ \ i\notin\Theta, \nonumber \\
			\Rightarrow \lambda^* &= \tilde{\lambda}^* + \dfrac{\sum_{i\in\Theta}\left(\frac{\tilde{\lambda}^*-2\alpha_iP_i-\beta_i}{2\alpha_i}\right)}{\sum_{i\notin\Theta}\frac{1}{2\alpha_i}},
	\end{align}}\normalsize
	where $\tilde{\lambda}^*=\dfrac{P_\text{tot}+\sum_{i=1}^N\frac{\beta_i}{2\alpha_i}}{\sum_{i=1}^N\frac{1}{2\alpha_i}}$ is the incremental cost for a related ED problem without generation constraints \eqref{eq:lambda_without}  (see, e.g., \cite{chen2016distributed, feng2017finite} for more details). This relationship between $\tilde{\lambda}^*$ and $\lambda^*$ is utilized in the main algorithm proposed in the paper to address the ED problem with generation constraints.
	
	\vspace{-.5em}\section{Distributed FxTS algorithm}\label{sec:without}
	\subsection{Without generation constraints}
	We first present our main results on solving distributed ED problem without generation constraints in a fixed time. The approach is based on designing a fixed-time consensus protocol on incremental costs $\{\lambda_i\}$, such that for the average consensus, \eqref{eq:P_star_without} is satisfied. Any node at which several components of the power system, such as generator and loads are connected, is referred as a \emph{bus} in electrical parlance. To this end, we make the following assumptions on the communication topology.
	\vspace{-.5em}
	\begin{Assumption}
		Communication topology between the generator buses $A(t)$ is connected and undirected for all $t\geq 0$.
	\end{Assumption}\vspace{-1em}
	\begin{Assumption}
		Each generator bus can exchange information only with its neighboring bus.
	\end{Assumption}

	% \textbf{WHAT DOES "BETWEEN" TWO GENERATOR BUSES MEAN HERE? THE NOISE IS ONLY ACTING AT THE INDIVIDUAL NODE $i$.I THINK THE EXPLANATION GIVEN AFTER (10) IS SUFFICIENT AND CORRECT. MAYBE WE CAN SIMPLY MOVE THIS ASSUMPTION AFTER THAT TEXT, AND SAY THAT THE NOISE HAS THESE PROPERTIES, AGAIN, INSTEAD OF TRYING TO SAY WHAT THAT NOISE REPRESENTS OR WHERE IT'S COMING FROM.}
	% \noindent\textit{\bf Assumption 1}: Communication topology between the generator buses $A(t)$ is connected and undirected at all times.\\
	% \noindent\textit{\bf Assumption 2}: Each generator bus can exchange information only with its neighboring bus. Furthermore, uncertainty arising from noisy communication between two generator buses is zero-mean and uniformly bounded.
	
	\noindent The active power $P_i$ for the $i^\text{th}$ generator is updated as:\vspace{-.5em}
	{\small\begin{align}\label{eq:Pi_dyn}
			\dot P_i(t) &=  p\sum_{j\in\mathcal{N}_i(t)}\big(\sgn(\lambda_j(t)-\lambda_i(t))+\sgn^{\mu_1}(\lambda_j(t)-\lambda_i(t)) \nonumber\\
			& \quad + \sgn^{\mu_2}(\lambda_j(t)-\lambda_i(t))\big) + \omega_i(t),
	\end{align}}\normalsize 
	with $p>0, 0<\mu_1<1, \mu_2>1$ and $P_i(0) = \sum_{k=1}^md_{i_k}P_{\text{L}_k}$. Constants $\mu_1, \mu_2$ are chosen such that the functions $\sgn^{\mu_1}(\cdot)$ and $\sgn^{\mu_2}(\cdot)$ are odd in their arguments. The function $\omega_i:\mathbb R_+\rightarrow \mathbb R$ models the uncertainty arising at the $i^\text{th}$ bus during computation of active power dispatch using \eqref{eq:Pi_dyn}. We make the following assumption on the noise $\omega_i$.
	\vspace{-.5em}
	
	\begin{Assumption}
		Additive noise $\omega_i$ is zero-mean and uniformly bounded for each $i = \{1, 2, \dots, N\}$.
	\end{Assumption}
	\vspace{-.5em}
	Note that the communication topology is allowed to vary with time, and thus the neighborhood set $\mathcal{N}_i(\cdot)$ is a function of time. It is assumed that there are $m$ load buses in the power system, and the quantity $P_{\text{L}_k}$ denotes the power demanded by the $k^\text{th}$ load bus, and $d_{i_k}$ represents the binary association between the generator bus $i$ and the load bus $k$, defined as:\vspace{-.5em}
	{\small\begin{equation*}
			d_{i_k} = \left\{
			\begin{array}{ll}
				1, & \text{if buses $i$ and $k$ are neighbors},\\
				0, & \text{otherwise}.
			\end{array}
			\right.
	\end{equation*}}\vspace{-1em}
	
	\begin{Remark}\label{rem1}
		Inclusion of $\{d_{i_k}\}$ ensures that any load bus is required to communicate its power demand only with its nearest generator bus, i.e., $\sum_{i=1}^Nd_{i_k}=1$ for all $k$. Therefore, $\sum_{i=1}^N\sum_{k=1}^md_{i_k}P_{\textup{L}_k} = P_\textup{tot}$.
		Furthermore, from \eqref{eq:Pi_dyn}, Assumption~2 and Lemma \ref{lemma:sign}, it can be shown that\vspace{-.5em}
		\begin{align*}
			\mathbb{E}\left[\sum\limits_{i=1}^N\dot{P}_i(t)\right] = 0 \ \Rightarrow \mathbb{E}\left[\sum\limits_{i=1}^N{P}_i(t)\right] = \sum\limits_{i=1}^N{P}_i(0) \triangleq P_\textup{tot}.
		\end{align*}
		Thus the update law \eqref{eq:Pi_dyn} ensures that the load balance constraint is satisfied at all times.
	\end{Remark}
	In what follows, we omit the time-variable $t$. The incremental cost associated with generator bus $i$ is updated as:
	{\small\begin{align}\label{eq:lambda_dyn}
			\frac{\dot\lambda_i}{2\alpha_i} &  = \dot P_i + \sgn^{\nu_1}\left(P_i-\frac{\lambda_i-\beta_i}{2\alpha_i}\right) + \sgn^{\nu_2}\left(P_i-\frac{\lambda_i-\beta_i}{2\alpha_i}\right)\nonumber \\ 
			&= p\sum\limits_{j=1}^N\!a_{ij}\big[\sgn(\lambda_j\!-\!\lambda_i)+\sgn^{\mu_1}(\lambda_j\!-\!\lambda_i) + \sgn^{\mu_2}(\lambda_j\!-\!\lambda_i)\big]\nonumber\\
			&\quad  \!+ \!\sgn^{\nu_1}\!\left(\!P_i\!-\!\frac{\lambda_i-\beta_i}{2\alpha_i}\!\right) \!+\sgn^{\nu_2}\!\left(\!P_i\!-\!\frac{\lambda_i-\beta_i}{2\alpha_i}\!\right) \!+ \!\omega_i,
	\end{align}}\normalsize
	where $0<\nu_1<1$ and $\nu_2>1$. As before, constants $\nu_1, \nu_2$ are chosen such that the functions $\sgn^{\nu_1}(\cdot)$ and $\sgn^{\nu_2}(\cdot)$ are odd in their arguments. Note that the update laws \eqref{eq:Pi_dyn} and \eqref{eq:lambda_dyn} for scheduled dispatch values and incremental costs only require information from the local bus and its neighboring buses. Thus, the proposed approach is fully distributed. We now show that under the proposed update laws, $\{\lambda_i\}$ and $\{P_i\}$ converge to their optimal values in a fixed-time even in the presence of additive uncertainty $\omega_i$.
	
	\begin{Theorem}\label{thm1}
		Let the update equations for scheduled dispatch and incremental costs be given by \eqref{eq:Pi_dyn}-\eqref{eq:lambda_dyn}. Then, there exists $T_1<\infty$ such that for all $t\geq T_1$, $P_i(t) = \dfrac{\lambda_i(t)-\beta_i}{2\alpha_i}$, for all $\lambda_i(0)\in \mathbb R$, $i\in \{1,2,\cdots,N\}$.
	\end{Theorem}
	\begin{proof}
		Let $e_i = P_i-\frac{\lambda_i-\beta_i}{2\alpha_i}$ for all $i\in\{1,\dots,N\}$ and $V = \dfrac{1}{2}\sum_{i=1}^Ne_i^2$ whose time derivative along \eqref{eq:Pi_dyn}-\eqref{eq:lambda_dyn} reads\vspace{-.5em}{\small
			\begin{align*}
				\dot V & = \sum_{i = 1}^Ne_i\left(\dot P_i-\frac{\dot \lambda_i}{2\alpha_i}\right) = -\sum_{i = 1}^N|e_i|^{1+\nu_1} -\sum_{i = 1}^N|e_i|^{1+\nu_2}\\
				& \overset{\eqref{eq:ineq}}{\leq} -\left(\sum\nolimits_{i = 1}^N|e_i|^2\right)^{\frac{1+\nu_1}{2}} -\frac{1}{N^{\frac{\nu_2-1}{2}}}\left(\sum\nolimits_{i = 1}^N|e_i|^2\right)^{\frac{1+\nu_2}{2}}\\
				& =  -2^\frac{1+\nu_1}{2}V^\frac{1+\nu_1}{2}-\frac{2^\frac{1+\nu_2}{2}}{N^{\frac{\nu_2-1}{2}}}V^\frac{1+\nu_2}{2}.
		\end{align*}}\normalsize
		Hence, per Lemma \ref{lemma:FxTS}, there exists $T_1<\infty$ satisfying\vspace{-.5em}
		{\small\begin{align}\label{eq:T1_bound}
				T_1\leq \frac{2}{2^\frac{1+\nu_1}{2}(1-\nu_1)}+\frac{2N^\frac{\nu_2-1}{2}}{2^\frac{1+\nu_2}{2}(\nu_2-1)},
		\end{align}}\normalsize
		such that for all $t\geq T_1$, $V(t) = 0$, i.e., $P_i(t) = \frac{\lambda_i(t)-\beta_i}{2\alpha_i}$ for all $i\in \{1, 2, \cdots, N\}$, independent of $\{\lambda_i(0)\}$.
	\end{proof}
	The following theorem shows that the buses reach consensus on the incremental costs $\{\lambda_i\}$ in a fixed-time in the presence of additive uncertainty $\omega_i$, resulting in generator powers attaining their optimal values, per discussion in Section~\ref{subsec:without}.
	
	\begin{Theorem}[{\bf Fixed-time consensus}]\label{thm2}
		Let $p\geq 2\Delta\sqrt{\frac{N\alpha_\text{max}}{\Lambda_2(L_A)\alpha_\text{min}}}$, where $\alpha_\text{max} = \max_i\alpha_i$, $\alpha_\text{min}=\min_i\alpha_i$, $\Gamma = 1\big/ \sum_{i = 1}^N\frac{1}{2\alpha_i}$ and $\Delta = \sup_{t\geq 0}\max_{i}|\omega_i(t)-\frac{\Gamma}{2\alpha_iN}\sum_{j=1}^N\omega_j(t)|$. Then, under the effect of update laws described by \eqref{eq:Pi_dyn}-\eqref{eq:lambda_dyn}, there exists $T_2<\infty$, such that for all $t\geq T_1+T_2$, $\lambda_i(t) = \lambda_j(t)$ for all $i\neq j$, $i,j\in\{1,\dots,N\}$, where $T_1$ satisfies \eqref{eq:T1_bound}. Consequently, under the effect of update laws \eqref{eq:Pi_dyn}-\eqref{eq:lambda_dyn}, $P_i$ converge to the optimal solution of \eqref{eq:EDP} without generation constraints within a fixed time $t\leq T_1+T_2$, even in the presence of additive uncertainty $\omega_i$.
	\end{Theorem}
	\begin{proof}
		From Theorem \ref{thm1}, we obtain that for any $t\geq T_1$, $P_i(t) = \frac{\lambda_i(t)-\beta_i}{2\alpha_i}$ for all $i=1,\dots,N$. Thus, \eqref{eq:lambda_dyn} reduces to\vspace{-.5em}
		{\small\begin{align}\label{eq:lambda_dot_aft_T1}
				\frac{\dot\lambda_i}{2\alpha_i} &=  p\sum_{j \in \mathcal N_i}(\sgn(\lambda_j-\lambda_i)+\sgn^{\mu_1}(\lambda_j-\lambda_i) \nonumber\\
				&\quad + \sgn^{\mu_2}(\lambda_j-\lambda_i))+\omega_i, \ \ \text{for all} \ t\geq T_1.
		\end{align}}\normalsize
		In what follows, we will only consider trajectories of $\lambda_i(t)$ for $t\geq T_1$. Let $\bar{\lambda}$ denote the average of $\{\lambda_i\}$ weighted by the inverse of the corresponding cost coefficients $\{\alpha_i\}$, i.e., $\label{eq:lambda_bar}\bar \lambda = \dfrac{\Gamma}{N}\sum_{i=1}^N\dfrac{\lambda_i}{2\alpha_i}$, where $ \Gamma = 1\big/ \sum_{i = 1}^N\frac{1}{2\alpha_i}$. From Lemma \ref{lemma:sign} and \eqref{eq:lambda_dot_aft_T1}, it follows that $\dot{\bar\lambda} = \frac{\Gamma}{N}\sum_{i = 1}^N\frac{\dot \lambda_i}{2\alpha_i} = \frac{\Gamma}{N}\sum_{i = 1}^N\omega_i$.
		
		We now show that $\lambda_i(t)=\bar{\lambda}$ for all $t\geq T_1+T_2$, where $T_2~<~\infty$. To this end, we define the consensus error $\tilde{\lambda}_i = \lambda_i-\bar{\lambda}$, and consider the candidate Lyapunov function $V = \frac{1}{2}\sum_{i = 1}^N\frac{1}{2\alpha_i}\tilde\lambda_i^2$. Its time derivative along \eqref{eq:lambda_dot_aft_T1} reads\vspace{-.5em}
		\begin{align*}
			\dot V &= \sum\nolimits_{i = 1}^N\tilde\lambda_ip\sum\nolimits_{j=1}^Na_{ij}\big(\sgn(\lambda_j-\lambda_i) +\sgn^{\mu_1}(\lambda_j-\lambda_i)\\
			&\quad + \sgn^{\mu_2}(\lambda_j-\lambda_i)\big) + \sum\nolimits_{i = 1}^N\!\tilde\lambda_i\left(\omega_i-\frac{\Gamma}{2\alpha_iN}\sum\nolimits_{j = 1}^N\!\omega_j\!\right)\\
			&\quad  = \dot V_1 + \dot V_2,
		\end{align*}
		where $\dot V_1 = p\sum_{i,j= 1}^N\tilde{\lambda}_ia_{ij}(\sgn(\lambda_j-\lambda_i) +\sgn^{\mu_1}(\lambda_j-\lambda_i) + \sgn^{\mu_2}(\lambda_j-\lambda_i))$ and $\dot V_2 = \sum_{i = 1}^N\tilde\lambda_i(\omega_i-\frac{\Gamma}{2\alpha_iN}\sum_{j=1}^N\omega_j)$. Since $\omega_i$ and $\alpha_i$ are bounded, we can bound $\dot V_2$ as \vspace{-.5em}
		\begin{align*}
			\dot V_2 \leq \sum_{i=1}^N|\tilde\lambda_i|.\left|\omega_i-\frac{\Gamma}{2\alpha_iN}\sum_{j=1}^N\omega_j\right|\leq \Delta \sum_{i = 1}^N|\tilde\lambda_i| = \Delta\|\tilde \lambda\|_1,
		\end{align*}
		where $\Delta = \sup_{t\geq 0}\max_{i}|\omega_i(t)-\frac{\Gamma}{2\alpha_iN}\sum_{j=1}^N\omega_j(t)|$ and $\tilde\lambda = \begin{bmatrix}\tilde\lambda_1 & \cdots& \tilde\lambda_N\end{bmatrix}^T$. Hence, we have that \vspace{-.5em}
		{\small\begin{align}\label{eq:V2_dot}
				\dot V_2 \leq \Delta \|\tilde\lambda\|_1 \leq \Delta\sqrt{N}\|\tilde\lambda\|_2 \leq 2\Delta\sqrt{N\alpha_\text{max}}V^\frac{1}{2},
		\end{align}}\normalsize
		where $\alpha_\text{max} = \max_i\alpha_i$. Now, consider the term $\dot V_1$. Note that $\lambda_i-\lambda_j = \tilde\lambda_i-\tilde\lambda_j$ and denote $\tilde\lambda_{ij} = \tilde\lambda_i-\tilde\lambda_j$. Thus, the term $\dot{V}_1$ can be rewritten as:\vspace{-.5em}
		{\small\begin{align}\label{eq:V1_dot-1}
				\dot V_1 & = p\sum\nolimits_{i,j= 1}^N\tilde\lambda_ia_{ij}(\sgn(\tilde \lambda_{ji}) +\sgn^{\mu_1}(\tilde \lambda_{ji}) + \sgn^{\mu_2}(\tilde \lambda_{ji}))\nonumber\\
				& \overset{\eqref{eq:f_ij_eij}}{=} \frac{p}{2}\sum\nolimits_{i,j= 1}^N\tilde\lambda_{ij}a_{ij}(\sgn(\tilde \lambda_{ji}) +\sgn^{\mu_1}(\tilde \lambda_{ji}) + \sgn^{\mu_2}(\tilde \lambda_{ji}))\nonumber\\
				& \overset{\eqref{eq:ineq}}{\leq} -\frac{p}{2}\left(\!\sum\nolimits_{i,j=1}^Na_{ij}\tilde{\lambda}_{ij}^2\!\right)^{\!\frac{1}{2}} -\frac{p}{2}\left(\!\sum\nolimits_{i,j=1}^Na_{ij}\tilde{\lambda}_{ij}^2\!\right)^{\!\!\!\!\frac{1+\mu_1}{2}} \nonumber \\
				&\quad -\frac{p}{2N^{\mu_2-1}}\!\left(\!\sum\nolimits_{i,j=1}^Na_{ij}\tilde{\lambda}_{ij}^2\!\right)^{\!\!\!\!\frac{1+\mu_2}{2}}
		\end{align}}\normalsize
		Note that $\sum_{i=1}^N\tilde{\lambda}_i=0$ and $\sum_{i,j=1}^Na_{ij}\tilde{\lambda}_{ij}^2=2\tilde{\lambda}L_A\tilde{\lambda}$. More-over, from Lemma \ref{lemma:Laplacian}, we conclude that
		\[
		\quad 4\alpha_\text{min}\Lambda_2(L_A)V \quad \leq \quad 2\Lambda_2(L_A)\tilde{\lambda}^T\tilde{\lambda} \quad \leq  2\tilde{\lambda}L_A\tilde{\lambda},
		\]
		where $\alpha_\text{min}=\min_i\alpha_i$. Thus, \eqref{eq:V1_dot-1} can be rewritten as:
		{\small\begin{align}\label{eq:V1_dot-2}
				\dot V_1 & \leq -p(\alpha_\text{min}\Lambda_2(L_A)V)^\frac{1}{2} - p2^{\mu_1}(\alpha_\text{min}\Lambda_2(L_A)V)^\frac{1+\mu_1}{2}\nonumber\\
				& \quad - \frac{p2^{\mu_2}}{N^{\mu_2-1}}(\alpha_\text{min}\Lambda_2(L_A)V)^\frac{1+\mu_2}{2}.
		\end{align}}\normalsize
		On combining \eqref{eq:V2_dot} and \eqref{eq:V1_dot-2}, $\dot V$ can be bounded as:{\small
			\begin{align*}
				\dot V &\leq -\left(p\sqrt{\Lambda_2(L_A)\alpha_\text{min}}-2\Delta\sqrt{N\alpha_\text{max}}\right)V^\frac{1}{2}\\
				& \quad -\underbrace{p2^{\mu_1}(\Lambda_2(L_A)\alpha_\text{min})^\frac{1+\mu_1}{2}}_{c_1}V^\frac{1+\mu_1}{2}\\
				& \quad -\underbrace{\frac{p2^{\mu_2}}{N^{\mu_2-1}}(\Lambda_2(L_A)\alpha_\text{min})^\frac{1+\mu_2}{2}}_{c_2}V^\frac{1+\mu_2}{2}.
		\end{align*}}\normalsize
		Thus, with $p\geq 2\Delta\sqrt{\frac{N\alpha_\text{max}}{\Lambda_2(L_A)\alpha_\text{min}}}$, we obtain that $\dot V\leq -c_1V^\frac{1+\mu_1}{2}-c_2V^\frac{1+\mu_2}{2}$.
		Hence, per Lemma \ref{lemma:FxTS}, we conclude that there exists $T_2$ satisfying  $T_2\leq \frac{2}{c_1(1-\mu_1)}+\frac{2}{c_2(\mu_2-1)}$,
		such that $V(t) = 0$ for all $t\geq T_1+T_2$, or equivalently, $\lambda_i = \lambda_j = \bar\lambda$ for all $t\geq T_1 + T_2$ and $i,j\in\{1,\dots,N\}$. 
	\end{proof}
	
	\vspace{-.5em}
	\begin{Remark}\label{rem2}
		From Theorems \ref{thm1} and \ref{thm2}, it follows that \eqref{eq:P_star_without}-\eqref{eq:lambda_without} are satisfied for all $t\geq T_1+T_2$. Thus, the update laws \eqref{eq:Pi_dyn}-\eqref{eq:lambda_dyn} solve the ED problem without generation constraints in a fixed-time $\bar{T}=T_1+T_2$. Moreover, Theorem~\ref{thm2} guarantees robustness to the additive disturbance $\{\omega_i\}$.
	\end{Remark}\vspace{-.5em}
	There may exist communication link failures or additions among generator buses, which results in a time-varying communication topology. We model the underlying graph $G(t) = (A(t),\mathcal{V})$ through a switching signal $\rchi(t):\mathbb{R}_+\to\Psi$ as $G(t)=G_{\rchi(t)}\coloneqq(A_{\rchi(t)},\mathcal{V})$, where $\Psi=\{1,2,\dots,R\}$ is a finite set consisting of index numbers associated to specific adjacency matrices $A(t) = [a_{ij}(t)]\in \{A_1,\dots, A_R\}$. Here, the function $\rchi$ is a piecewise constant, right-continuous function of time. Let $t_0,t_1,\dots$ be the switching time sequence characterized by changes in information flow. For any time $t\in\left[t_i,t_{i+1}\right)$, the topology with adjacency matrix $A_{\rchi(t)}$ is active. The corollary below explores the impact of switching on the fixed-time convergence guarantees.
	% \vspace{-.5em}
	\begin{Cor}[{\bf Time-varying topology}]\label{cor1}
		Let the underlying topology for any time interval $\left[t_i,t_{i+1}\right)$ be connected. Then for the switching topology scenario defined by $\rchi(t)$, Theorems~\ref{thm1} and \ref{thm2} continue to hold for the ED problem without generation constraints.
	\end{Cor}
	\begin{proof}
		Note that $V$ in the proof of Theorem \ref{thm2} is a common Lyapunov function for \eqref{eq:Pi_dyn}-\eqref{eq:lambda_dyn}, under an arbitrary commutation among the set of connected graphs. Only place where the underlying network topology shows up explicitly is in \eqref{eq:V1_dot-2}, and consequently in the expression for the settling time in Theorem \ref{thm2}. Let $\Lambda_2^*$ denote the minimum of the second smallest eigenvalues of all graph Laplacians of the associated adjacency matrices, i.e., $\Lambda_2^*=\min\limits_t \Lambda_2(L_{A(t)})$. Since $\Psi$ is a finite set, the minimum exists. Moreover, the underlying graph is always assumed to be connected, and thus $\Lambda_2^*>0$. Thus, the inequality in \eqref{eq:V1_dot-2} holds with $\Lambda_2(L_A)$ replaced by $\Lambda_2^*$, and thus Theorem \ref{thm2} holds with suitably modified settling-time coefficients $c_1$ and $c_2$.
	\end{proof}
	
	\noindent \textbf{A note on discrete-time implementation}: Continuous-time algorithms offer effective insights into designing accelerated schemes for distributed optimization. However, sampling and acquisition constraints render implementation of continuous-time algorithms impractical. This note explores discrete analog of \eqref{eq:Pi_dyn}-\eqref{eq:lambda_dyn}, such that the resulting discrete-time dynamics of scheduled power dispatch and incremental costs are \emph{practically} fixed-time stable. The origin of a discrete-time dynamical system with state variable $x(\cdot)$ is globally practically fixed-time stable if for every $\epsilon>0$ there exists $N_\epsilon\in\mathbb{N}$ such that any solution $x(\cdot,x_0)$ satisfies $\|x(k,x_0)\|\leq\epsilon$ for $k\geq N_\epsilon$ independently of $x_0$ \cite{poly2018consistent}. Consider the continuous-time dynamical system of the form, $\dot{z} = -|z|z$, where $z\in\mathbb{R}$. The semi-implicit Euler-discretization scheme with step-size $h>0$ given by
	{\small\begin{align*}
			\frac{z(k\!+\!1)\!-\!z(k)}{h} = -|z(k)|z(k\!+\!1) \quad \Rightarrow z(k\!+\!1) = \frac{z(k)}{1+h|z(k)|},
	\end{align*}}\normalsize
	renders origin practically fixed-time stable, since $|z(1)|\leq h^{-1}$ independently of $z_0\triangleq z(0)$, and $|z(k)|\leq(kh)^{-1}$. The general idea of attack for consistent discretization is to hybridize a fixed-time consensus scheme on incremental cost \eqref{eq:lambda_dyn} with a discrete-time update equation on scheduled power dispatch, such that similar to \eqref{eq:Pi_dyn}, the following two conditions are satisfied: (1) $\sum_{i=1}^NP_i(k)=P_\text{tot}$, and (2) origin of discrete-time dynamical system with state-variable $z_i(k)\triangleq P_i(k)-(\lambda_i(k)-\beta_i)/(2\alpha_i)$ is practically fixed-time stable for all $i\in\{1,\dots,N\}$. To this end, we consider the following discrete-time update-laws for $\{\lambda_i\}$ and $\{P_i\}$:\vspace{-.5em}
	\begin{subequations}\label{eq:discrete_update}
		{\small\begin{align}
				\frac{\lambda_i(k+1)}{2\alpha_i} &= (W_k)_i\lambda(k)-\lambda_i(k)+P_i(k)+\frac{\beta_i}{2\alpha_i} \nonumber \\
				&\quad - \dfrac{P_i(k)-\left(\frac{\lambda_i(k)-\beta_i}{2\alpha_i}\right)}{1+h\left|P_i(k)-\left(\frac{\lambda_i(k)-\beta_i}{2\alpha_i}\right)\right|},\\
				P_i(k+1) &= -c_k(L_A)_i\lambda(k) + P_i(k),
		\end{align}}\normalsize
	\end{subequations}
	where $W_k\triangleq I-c_kL_A$, $c_k=1/\Lambda_{(k-1)\textup{mod}K+2}$, and $K$ is the number of non-zero distinct eigenvalues of $L_A$. It is easy to observe that $\sum_{i=1}^NP_i(k)=P_\text{tot}$ with $P_i(0)=\sum_{j=1}^md_{i_j}P_{\text{L}_j}$, and $z_i(k+1)=z_i(k)/(1+h|z_i(k)|)$, i.e., $|z_i(k)|\leq\epsilon$ after $1/(h\epsilon)$ iterations. Moreover, the consensus on incremental cost variables occurs in finite iterations per FACA~\cite{kibangou2012graph}. Note that FACA is invoked again for calculating optimal dispatch in the constrained ED problem as a discrete analog for fixed-time average consensus (see step 6 in Algorithm~\ref{alg}). Note that the discrete consensus scheme is adopted from FACA, and is not a direct discretization of \eqref{eq:Pi_dyn}-\eqref{eq:lambda_dyn}. Thus, the discretized scheme does not account for robustness to additive disturbances and time-varying topology. A detailed investigation into direct discretization scheme is beyond the scope of the current work.

	\vspace{-1em}\subsection{With generation constraints}\label{sec:with}

	Since \eqref{eq:lambda_star_with} captures the relationship between incremental costs of the constrained and the unconstrained ED problems, \eqref{eq:EDP} can also be solved in a fixed-time using Algorithm \ref{alg}. First, the unconstrained ED problem is solved (Step~1). Then, exploiting the relationship between incremental costs of the constrained and the unconstrained ED problems, the set $\Theta$ is updated an incremental fashion. If the constraints are inactive for a given generator, the optimal incremental cost is related to optimal dispatch through equality constraint, otherwise the dispatch values are saturated at the generation limits (Step~4). Steps~5-6 aim to compute the numerator and denominator terms in \eqref{eq:lambda_star_with}. Since, both the numerator and denominator terms involve summation, this can be done by running fixed-time average consensus for the tuple $\{(y_i,z_i)\}$. The update law \eqref{eq:lambda_dot_aft_T1} with $\alpha_i=0.5$ for all $i$ simply defines the fixed-time average consensus scheme. Incremental costs and dispatch values are then updated in Steps~7-8.
	
	\vspace{-1em}\begin{algorithm}
		\caption{Distributed Fixed-time Algorithm for ED}
		\begin{algorithmic}[1]\label{alg}
			\renewcommand{\algorithmicrequire}{\textbf{Input:}}
			\renewcommand{\algorithmicensure}{\textbf{Output:}}
			\REQUIRE cost-coefficients $\{\alpha_i,\beta_i,\gamma_i\}$, load-demands $\{P_{\text{L}_k}\}$
			\ENSURE  Optimal incremental cost \& dispatch $(\lambda^*,\{P_i^*\})$
			\\ \textit{Initialization}: $\Theta=\emptyset$
			\STATE Run \eqref{eq:Pi_dyn}-\eqref{eq:lambda_dyn} to solve unconstrained ED and find $\big(\bar\lambda^*,\{P_i^*\}\big)$. Set $(\lambda,\{P_i\})\gets\big(\bar\lambda^*,\{P_i^*\}\big)$
			\WHILE {Generation constraint violations}
			\STATE  $\Omega\triangleq\{i\notin\Theta:(P_i<P_i^\text{min})\wedge (P_i>P_i^\text{max})\}$, $\Theta\gets\Theta\cup\Omega$\vspace{-.5em}
			\STATE Calculate optimal dispatch using
			\begin{small}{\begin{equation}\label{eq:P_comp}
						P_i\gets\left\{\begin{array}{cl} \frac{\lambda-\beta_i}{2\alpha_i}, & i\notin\Theta, \\
							P_i^\text{min} \ \text{or} \ P_i^\text{max}, & i\in\Theta.
						\end{array}\right.
			\end{equation}}\end{small}
			\STATE \begin{small}{$(y_i(0),z_i(0))\gets\left\{\begin{array}{cl}
					\left(\frac{\lambda-2\alpha_iP_i-\beta_i}{2\alpha_i},0\right), & i\in\Theta,
					\\
					\left(0,\frac{1}{2\alpha_i}\right), & i\notin\Theta.\end{array}\right.$}\end{small}
			\STATE Run \eqref{eq:lambda_dot_aft_T1} with $\alpha_i=0.5$ on $\{y_i,z_i\}$ to obtain $(y_c,z_c)$
			\STATE $\lambda\gets\lambda+{y_c}/{z_c}$
			\STATE Calculate optimal dispatch using \eqref{eq:P_comp}
			\ENDWHILE
			\RETURN $(\lambda,\{P_i\})$ 
		\end{algorithmic}
	\end{algorithm}
	
	\vspace{-1em}\begin{Remark}\label{rem4}
		Steps 6 and 7 in the algorithm are related to \eqref{eq:lambda_star_with}, where it follows that $y_c=(1/N)\sum_{i\in\Theta}y_i(0)$ and $z_c=(1/N)\sum_{i\notin\Theta}z_i(0)$. In addition, since the number of generator buses are finite, it follows that the number of times the \textup{\texttt{While}}-loop gets called is also finite. Thus, the entire algorithm gets executed within a fixed-time.
	\end{Remark}
	
	\vspace{-1.5em}
	
	\section{Case Studies}\label{sec:examples}
	We now present numerical examples involving IEEE test cases. Simulation parameters in Theorems \ref{thm1}-\ref{thm2} are chosen as: $\mu_1=\nu_1=0.8$, $\mu_2 = \nu_2 = 1.2$, $p = 1485$ for IEEE-57 bus case. Unless stated otherwise, solid lines in all the example scenarios indicate true power dispatch from generators while dotted lines indicate optimal dispatch values.
	
	\begin{figure}[ht!]
		\centering
		\includegraphics[height=1.5in,width=1\columnwidth,clip]{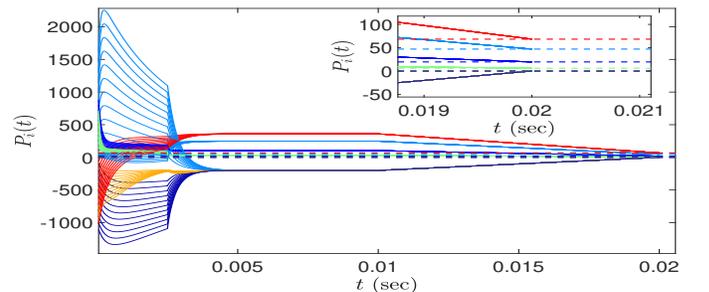}
		\vspace{-2em}
		\caption{Generator power $P_i$ for 57 bus example with generation constraints and switching topology for various initial conditions $P_i(0)$.}
		\label{fig:switch}
		\vspace{-1em}
	\end{figure}

	\vspace{-1em}\subsection{Switching communication topology}\label{subsec:case3}
	This case study concerns with ED problem for a 57-bus system with 7 generator buses, and further incorporates switching communication topology between generator buses. In particular, the communication topology between the generator buses in the specified 57-bus system is switched randomly in every 0.0025 seconds between randomly generated connected graphs. The parameters for cost functions are adopted from \cite{zimmerman2010matpower}. Figure~\ref{fig:switch} shows the convergence behavior of generator power dispatch under switching communication topology for various initial conditions.
	
	\vspace{-1em}\subsection{Time-varying demand and uncertain information}\label{subsec:case4}
	In this case study, a zero mean Gaussian communication disturbance is considered with variance of $0.01$. Additionally, the net load demand in the beginning is 141.13~MW, which then alternates between 69.83~MW and 212.81~MW at time instants 0.66s, 1.1s, 1.31s and 1.75s. Figure~\ref{fig:load_dist_57} shows the scheduled power dispatch from generators as the net load is varied with time. It can be seen that the generators rapidly adjust to variability in total load demand, and converge to optimal dispatch values in a fixed time.
	
	\begin{figure}[ht!]
		\centering
		\includegraphics[width=1\columnwidth,clip]{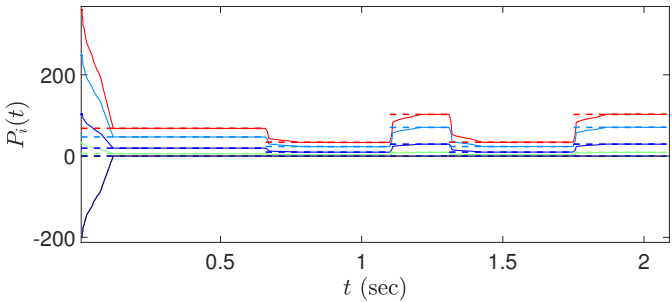}
		\vspace{-2em}
		\caption{Generator power $P_i$ for 57 bus example with generation constraints, communication disturbance and time varying load.}
		\label{fig:load_dist_57}
		\vspace{-1.5em}
	\end{figure}
	
	\vspace{-1em}\subsection{Convergence performance comparison}\label{subsec:comp}
	In this case study, we evaluate the performance of the discretized implementation of our fixed-time ED algorithm against the ICC algorithm \cite{zhang2012convergence}. For ease of illustration, the two algorithms are evaluated on the IEEE-30 bus network comprising of six generators. Figure~\ref{fig:comparison} shows the performance of the two algorithms for a net load demand of 250 MW, under a constant step-size of 0.1s for discretized implementation. As can be seen in Figure~\ref{fig:comparison}, the ICC algorithm requires nearly 5000 iterations for convergence, while the proposed fixed-time ED algorithm converges under 20 iterations. This super-accelerated convergence of our method is observed despite it being a fully distributed algorithm, whereas the ICC algorithm assumes a single leader node that aggregates information from every other node in the network. These results are consistent with the convergence behavior observed in \cite{chen2016distributed}, where both ICC algorithm, as well as a finite-time ED algorithms require more than 100$s$ (step-size 0.01$s$) for convergence.
	
	\begin{figure}[!ht]
		\centering
		\includegraphics[height=1.5in,width=1\columnwidth,clip]{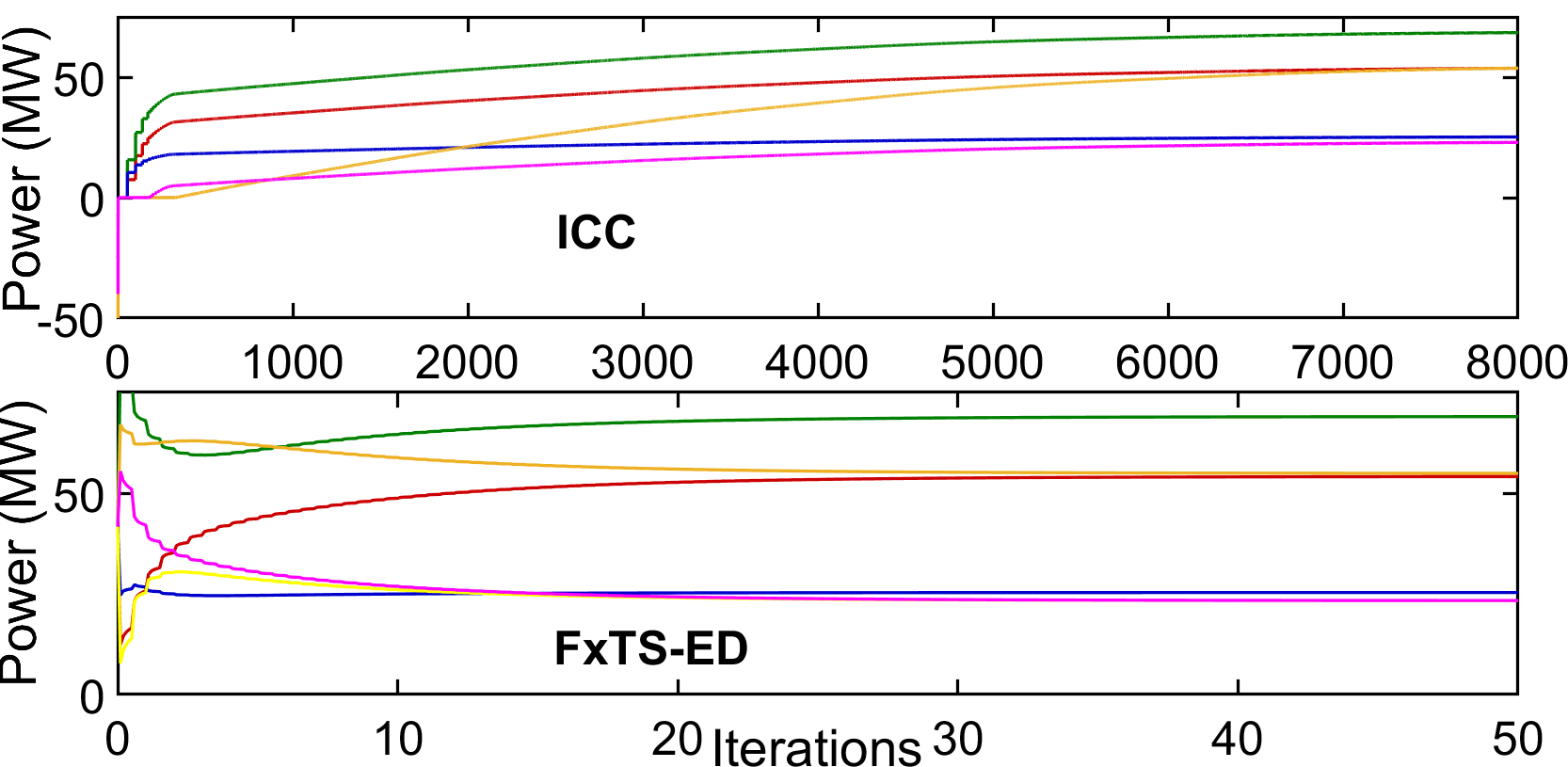}
		\vspace{-1em}
		\caption{Generator power $P_i$s for the incremental cost consensus (ICC) algorithm, and our fixed-time ED algorithm for IEEE-30 bus.}
		\label{fig:comparison}
		\vspace{-1.5em}
	\end{figure}
	
	% \begin{figure*}
	% 	\begin{center}
	% 		\begin{tabular}{cc}
	% 			\includegraphics[height=1.4in,width=1.31\columnwidth]{indi_power.png} & \includegraphics[height=1.4in,width=0.65\columnwidth]{demand.png} \cr
	% 			(a)  & (b) 
	% 		\end{tabular}
	% 		\vspace{-1em}
	% 		\caption{Comparison of our fixed-time economic dispatch algorithm with the incremental cost consensus (ICC) algorithm on IEEE-30 bus. (a) Trajectories of generator power $P_i$ for the two algorithms, (b) Total power $\sum P_i$ produced by all the generators.}
	% 		\label{fig:comparison}
	% 	\end{center}
	% 	\vspace{-2em}
	% \end{figure*}
	
	\vspace{-1em}
	\section{Conclusion}\label{sec:Conclusion}
	A novel, fixed-time convergent, distributed algorithm for solving constrained economic dispatch problem subject to communication uncertainties and time-varying topology is proposed. The algorithm is evaluated on standard IEEE test cases for several challenging scenarios ranging from unconstrained ED problem to constrained ED problem with time-varying load and communication topology, and it shown that algorithm exhibits accelerated convergence behavior. A discretization scheme is also suggested that renders the discrete-time implementation of the proposed continuous algorithm practically fixed-time convergent.\vspace{-1em}
	
	\bibliographystyle{IEEEtran}
	\bibliography{myreferences}
	
\end{document}